\documentclass[12pt, amscd]{amsart}
\usepackage{graphicx}
\usepackage{amsmath,amssymb, amsthm}
\usepackage{pstricks-add}

\newtheorem{thm}{Theorem}[section]
\newtheorem{cor}[thm]{Corollary}
\newtheorem{lem}[thm]{Lemma}

\newtheorem{conj}[thm]{Conjecture}

\theoremstyle{definition}
\newtheorem{defn}[thm]{Definition}

\newtheorem{exm}[thm]{Example}
\newtheorem{rem}[thm]{Remark}

\DeclareMathOperator{\girth}{girth}

\DeclareMathOperator{\hgt}{ht}

\DeclareMathOperator{\ass}{Ass}
\DeclareMathOperator{\lcm}{lcm}
\DeclareMathOperator{\pc}{\mathcal{PC}}
\DeclareMathOperator{\D}{\mathcal{D}}
\DeclareMathOperator{\G}{\mathcal{G}}
\DeclareMathOperator{\H1}{\mathcal{H}}

\DeclareMathOperator{\C}{\mathcal{C}}

\def\Z {\mathbb Z}
\def\N {\mathbb N}

\begin{document}

\title[Cohen-Macaulay oriented graphs with large girth] {Cohen-Macaulay oriented graphs with large girth}

\author[L.X. Dung]{Le Xuan Dung}
\address{Faculty of Natural Sciences, Hong Duc University,
No. 565 Quang Trung Street, Dong Ve Ward, Thanh Hoa, Vietnam}
\email{lxdung27@gmail.com}

\author[T.N. Trung]{ Tran Nam Trung}
\address{Institute of Mathematics, VAST, 18 Hoang Quoc Viet, Hanoi, Viet Nam.}
\email{tntrung@math.ac.vn}

\date{}
\commby{}
%-----------------------------------------------------------
\begin{abstract}
We classify the Cohen-Macaulay weighted oriented graphs whose underlying graphs have girth at least $5$. 
\end{abstract}
% -----------------------------------------------------------
\maketitle
% -----------------------------------------------------------
\section*{Introduction}

Throughout the paper let $R = K[x_1,\ldots,x_d]$ be a polynomial ring over a given field $K$. Let $\D$ be a digraph with the vertex set $V(\D)=\{x_1,\ldots,x_d\}$ and the edge set $E(\D)$. Let  a function $\omega \colon V(\D) \to \Z_{>0}$ which we call a {\it weight vertex} on $\D$.  The couple $(\D,\omega)$ is called  a {\it weighted oriented graph} (or oriented graph for short) and simply write as $\D$ if there is no confusion.  The edge ideal of $\D$ is defined by
$$I(\D) = (x_ix_j^{\omega(x_j)} \mid (x_i,x_j)\in E(\D)).$$
This ideal has an interesting connection with the Reed-Muller codes in the coding theory (see e.g. \cite{HLMRV, MPV}).

Recently, the algebraic properties and invariants of these ideals are studied by many authors (see e.g. \cite{CK, GMSVV, GMV, HLMRV, KN, MP, MPV, PRT, PRV,S, SS, SSTY,ZWT}). In this paper we study the Cohen-Macaulayness of these ideals. An oriented graph $\D$ is called {\it Cohen-Macaulay} (respectively, {\it unmixed}) if the quotient ring $R/I(\D)$ is Cohen-Macaulay (respectively, unmixed). The problem on classifying Cohen-Macaulay oriented graphs are done for forests by \cite{GMSVV}, for chordal graphs by \cite{S}, for bipartite graphs by \cite{HLMRV}, or more general, for very well-covered graphs by \cite{PRV}. In this paper, we enlarge the class of Cohen-Macaulay oriented graphs by classifying completely them when their underlying graphs have girth at least $5$.  

Before stating the the result, we recall some definitions and terminologies. Let $G$ be a simple undirected graph. The {\it girth} of $G$, denoted by $\girth (G)$, is the length of any shortest cycle in $G$ or in the case $G$ is a forest we consider the girth to be infinite. An edge, in a graph $G$, incident with a leaf is called the {\it pendant edge}. A $5$-cycle $C_5$, i.e. a cycle of lenght $5$, of  $G$ is called {\it basic}, if $C_5$ does not contain two adjacent vertices of degree three or more in $G$.  In a given graph $G$, let $C(G)$ denote the set of all vertices which belong to basic $5$-cycles and let $P(G)$ denote the set of vertices  which are incident with {\it pendant} edges in $G$. Then, $G$ is in the class $\mathcal{PC}$ if $V(G)$ can be partitioned into $V(G)=P(G) \cup C(G)$ and the pendant edges form a perfect matching of $G[P(G)]$ (see \cite{FHN}).  

\medskip

The main result of the paper is the following theorem.

\medskip

{\noindent {\bf Theorem \ref{main-theorem}}.  \it Let $\D$ be a weighted oriented graph with underlying graph $G$ of girth at least $5$. Then, the following conditions are equivalent:
\begin{enumerate}
\item $\D$ is Cohen-Macaulay.
\item $G$ is Cohen-Macaulay and $\D$ is unmixed.
\item $G$ is in the class $\pc$ and $\D$ satisfies:
\begin{enumerate}
\item For every pendant edge $xy$ of $G$ with $\deg_G(y)=1$ and $\omega(x) \ne 1$, $(y,x)\in E(\D)$ if $(z,x)\in E(\D)$ for some $z\in N(x)\setminus \{y\}$.
\item If $\C$ is an induced subgraph of $\D$ on a basic $5$-cycle of $G$, then
\begin{enumerate}
\item $\C$ is unmixed.
\item If $x$ is a vertex on $\C$ with $\deg_G(x) > 2$, then $\C\setminus x$ is unmixed.
\item If $x$ is a vertex on $\C$ such that $\omega(x)\ne 1$ and $(w,x)\in E(\D)$ for some $w\in V(\D)\setminus V(\C)$, then $(y,x), (v,x)\in E(\D)$, where $N_{\C}(x)=\{y,v\}$.
\end{enumerate}
\end{enumerate}
\end{enumerate}
}

\medskip

Here we explain some key points in the proof of the theorem \ref{main-theorem}. If $\D$ is Cohen-Macaulay,  so is $I(G) = \sqrt{I(\D)}$ by \cite[Theorem 2.6]{HTT}. Assume further that $G$ has girth at least $5$, then $G$ is in the class $\pc$. For the property $(a)$ we prove a general fact about unmixed monomial ideals based on their primary decompositions (see Lemma \ref{CM-L01}). Next we introduce the notion of reducible vertices on an induced subgraph of $\D$ on a basic $5$-cycle of $G$ to investigate the weight function on it. Together with the structure of $G$, we can prove the combinatorial property $(b)$.

Conversely, if $\D$ satisfies the condition $(3)$. In order to prove $\D$ is Cohen-Macaulay, it suffices to show that $(I(\D),x^m)$ and $I(\D)\colon x^m$ are Cohen-Macaulay where $x$ is a 
reducible vertex on some basic $5$-cycle $C$ of $G$  and $m=\omega(x)$. In fact, we  prove these ideals are Cohen-Macaulay by examining that they are edge ideals of some oriented graphs, and then show that these oriented graphs also satisfy the condition $(3)$, and the conclusion follows by induction.

% \medskip
%The paper consists of two sections.  In Section $1$, we set up some basic notations, terminologies for the graph. In Section $2$, we classify  Cohen-Macaulay weighted graphs of girth at least $5$.  

\section{Preliminaries}

In this section we collect some terminologies and results from the graph theory, the unmixedness of oriented graphs, and Cohen-Macaulay monomial ideals and their colon ideals. We start with some terminologies from the graph theory. Two vertices $u, v$ of a simple undirected graph $G$ are adjacent if $uv$ is an edge of $G$. An independent set in $G$ is a set of vertices no two of which are adjacent to each other. An independent set of maximum size will be referred to as a maximum independent set of $G$, and the independence number of $G$, denoted by $\alpha(G)$, is the cardinality of a maximum independent set in $G$. An independent set $S$ in $G$ is maximal (with respect to set inclusion) if the addition to $S$ of any other vertex in the graph destroys the independence.  If every maximal independent set of $G$ has the same size, namely $\alpha(G)$, the independence number of $G$,  then $G$ is {\it well-covered} (see \cite{P}).  A set $C\subseteq V(G)$ is a {\it vertex cover} of $G$ if every edge of $G$ has an end vertex in $C$.  A vertex cover is minimal if no its proper subset is still a vertex cover. Obviously,  $C$ is a vertex cover if and only if $V(G)\setminus C$ is independent set and vice versa.  

The neighborhood of a vertex $v$ of $G$ is the set $N_G(v) = \{u \mid u \in V(G) \text{ and } vu\in E(G)\}$. The {\it degree} of $v$ in $G$, denoted by $\deg_G(v)$, is the number of its neighbors, i.e. $\deg_G(v) = |N_G(v)|$. If $\deg_G(v) = 1$, $v$ is called a leaf. If an edge is incident with a leaf, it is called a pendant. The closed neighborhood of $v$ is  $N_G[v] = N_G(v) \cup \{v\}$; if there is no ambiguity on $G$, we use $N(v)$ and $N[v]$, respectively. 

For a vertex $v$ of $G$, let $G\setminus v = G\setminus\{v\}$ and $G_v=G\setminus N_G[v]$.  The vertex $v$ is called a {\it shedding vertex} of $G$ if no independent set in $G_v$ is a maximal independent set in $G\setminus v$. If $G$ is well-covered and $v$ is a shedding vertex of $G$, then $G\setminus v$ is well-covered and $\alpha(G\setminus v) =\alpha(G)$.  We next introduce the class of {\it vertex decomposable graphs} (see e.g. \cite{W}).  A graph $G$ is vertex decomposable  if $G$ is a totally disconnected graph (i.e. with no edges) or if it has a shedding vertex $v$ so that $G\setminus v$ and $G_v$ are both vertex decomposable. It is well-known that $G$ is Cohen-Macaulay whenever it is well-covered and vertex-decomposable.

Now assume that $G$ is the underling graph of a oriented graph $\D$. If the weight $\omega$ on $V(G)$ is the trivial one, i.e., $\omega(v)=1$ for all $v\in V(G)$, then $I(\D)$ is just the usual edge ideal of $G$ and denoted by $I(G)$. It is well-known that (see e.g. \cite[Proposition 6.1.16]{Vi}):

\centerline{$\ass(R/I(G)) = \{(v\mid v \in C) \mid C \text{ is a minimal vertex cover of } G\}.$}
\medskip

\noindent In particular,  $\dim R/I(G) = \alpha(G)$ whenever $V(G)=\{x_1,\ldots,x_d\}$. The graph  $G$ is called a Cohen-Macaulay graph if the ring $R/I(G)$ is Cohen-Macaulay. Accordingly,  if $G$ is Cohen-Macaulay, it is well-covered.

Observe that $I(G)=\sqrt{I(\D)}$, thus $\D$ is Cohen-Macaulay if so is $G$ (due to \cite[Theorem 2.6]{HTT}). By this fact, in order to investigate the Cohen-Macaulayness of an oriented graph we first need the structure of the underlying Cohen-Macaulay graph and then look for weight vertices on it. In this paper we focus on graphs of girth at least $5$ and so the following result plays a crucial role in the paper (see \cite[Theorem 20]{BC} or \cite[Theorem 2.4]{HMT}).

\begin{lem}  \label{HMT} Let $G$ be a connected graph of girth at least $5$. Then, the following statements are equivalent:
\begin{enumerate}
\item $G$ is well covered and vertex decomposable;
\item $G$ is Cohen-Macaulay;
\item $G$ is either a vertex or in the class $\mathcal{PC}$.
\end{enumerate}
\end{lem}

A monomial ideal $I$ is {\it unmixed} if every its associated prime has the same height.  It is well known that $I$ is unmixed if $R/I$ is Cohen-Macaulay.  If $I(\D)$ is unmixed,  we say that $\D$ is unmixed.  Note that $I(G)=\sqrt{I(\D)}$, so $G$ is well-covered if $\D$ is unmixed.  

In order to investigate the unmixed oriented graph $\D$ we use a criterion given in \cite{PRT}. First we recall the following definition.

\begin{defn} Let $G$ be the underlying graph of an oriented graph $\D$. For a vertex cover $C$ of $G$, define
\begin{align*}
L_1(C) &= \{x\in C \mid (x,y) \in E(\D) \text { for some } y\notin C\},\\
L_2(C) &=\{x\in C\mid x\notin L_1(C) \text{ and } (y,x)\in E(\D) \text{ for some } y\notin C\},\\
L_3(C) &= \{x\in C\mid N_G(x)\subseteq C\} = C \setminus (L_1(C) \cup L_2(C)).
\end{align*}

A vertex cover $C$ of $G$ is called a {\it strong vertex cover} of $\D$ if either $C$ is a minimal vertex cover of $G$ or for all $x\in L_3(C)$ there is $(y, x)\in E(\D)$ such that $y\in L_2(C)\cup L_3(C)$ with $\omega(y)\geqslant 2$. 
\end{defn}

\begin{lem} \cite[Theorem 31]{PRT}\label{unmixedness} $I(\D)$ is unmixed if and only if $G$ is well-covered and $L_3(C)=\emptyset$ for every strong vertex cover $C$ of $\D$.
\end{lem}

\medskip

We now state a technique to study the Cohen-Macaulayness of monomial ideals.

\begin{lem}\label{CM-Q} Let $I$ be a monimial ideal and $f$ a monomial not in $I$.  We have 
\begin{enumerate}
\item If $I$ is Cohen-Macaulay, then $I\colon f$ is Cohen-Macaulay.
\item If $I\colon f$ and $(I,f)$ are Cohen-Macaulay with $\dim R/I\colon f = \dim R/(I,f)$, then $I$ is Cohen-Macaulay.
\end{enumerate}
\end{lem}
\begin{proof} From the exact sequence 
$$0 \to R/(I\colon f) \to R/I \to R/(I,f)\to 0$$
we obtain
$$\dim(R/I) = \max\{\dim R/(I\colon f), \dim R/(I,f)\}.$$

Hence, the lemma follows from this fact and \cite[Lemma 4.1 and Theorem 4.3]{CHHKTT}.
\end{proof}

\begin{lem} \label{dim} Let $G$ be a well-covered graph.  If $v$ is a shedding vertex of $G$, then 
$$\dim R/I(G) = \dim R/(I(G\setminus v),v)=\dim R/I(G):v.$$
\end{lem}
\begin{proof} We may assume that $V(G)=\{x_1,\ldots,x_d\}$ so that $\dim R/I(G)=\alpha(G)$. Since $v$ is a shedding vertex and $G$ is well-covered, we have $\alpha(G)=\alpha(G\setminus v)$. Hence,
\begin{align*}
\dim R/(v,I(G)) &= \dim R/(v,I(G\setminus v)) = \dim K[V(G\setminus v)]/ I(G\setminus v) \\
&= \alpha(G\setminus v) = \alpha(G) = \dim R/I(G).
\end{align*}

Finally,  since $I(G)$ is unmixed and $v\notin I(G)$, we have $\hgt(I(G)) = \hgt(I(G)\colon v)$, and hence
$\dim R/(I(G):v) = \dim R/I(G)$,  and the lemma follows.
\end{proof}

\medskip

Let $G$ be the underlying graph of an oriented graph $\D$. A vertex $v$ of $\D$ is a source if $(v,x) \in E(\D)$ whenever $vx\in E(G)$ and it is a sink if $(x,v)\in E(G)$ whenever $vx\in E(G)$. Since the ideal $I(\D)$ does not change if we vary the weight of any source, so that we often assign weight $1$ to a source. With this convention, we have the following result which classifies Cohen-Macaulay oriented graphs whose underlying graphs have a perfect matching consisting of pendant edges.

\begin{lem}\label{H2} \cite[Theorem 1.1]{HLMRV} Let $\D$ be an oriented graph and let $G$ be its underlying graph. Suppose that $G$ has a perfect matching 
$\{x_1y_1,\ldots , x_ry_r\}$, where $y_i$'s are leaves in $G$. Then the following are equivalent:
\begin{enumerate}
\item $\D$ is Cohen-Macaulay;
\item $\D$ is unmixed;
\item for each edge $x_sy_s$ of $G$, if $(x_r,x_s)\in E(\D)$ for some $r$, then $(y_s,x_s)\in E(\D)$.
\end{enumerate}
\end{lem}

In order to write the lemma \ref{H2} for the unmixed path of length $3$, for an edge $xy$ of $G$ define
$$
m_{\D}(xy) = \begin{cases}
xy^{\omega(y)} & \text{ if } (x,y)\in E(\D),\\
yx^{\omega(x)} & \text{ if } (y,x) \in E(\D),
\end{cases}
$$
which is called the monomial corresponding to the edge of $\D$ induced on $xy$. Then,

\begin{cor}\label{corP4} If the underlying graph $G$ of $\D$ is a path of length $3$ with $E(G) = \{xy,yz,zv\}$, then $\D$ is unmixed if and only if it satisfies:
\begin{enumerate}
\item if $(y,z)\in E(\D)$ and $\omega(z)\ne 1$, then $(v,z)\in E(\D)$, or
\item if $(z,y)\in E(\D)$ and $\omega(y)\ne 1$, then $(x,y)\in E(\D)$.
\end{enumerate}
Or equivalently, $$\deg_y(m_{\D}(yz)) \leqslant \deg_y(m_{\D}(xy)) \text{ and } \deg_z(m_{\D}(yz)) \leqslant \deg_z(m_{\D}(vz)).$$
\end{cor}

\medskip

For convenient, we make use the following convention.

\begin{defn} For a subgraph $G'$ of $G$, the induced subgraph of $\D$ on $G'$ is the oriented graph $\D'$ defined by
\begin{enumerate}
\item $V(\D') = V(G')$ and $E(\D')=\{(x,y) \mid (x,y)\in E(\D) \text{ and }  xy \in E(G')\}$.
\item The weight function is the restriction of $\omega$ on $V(G')$.
\end{enumerate}
In the case $G'$ is an induced subgraph of $G$, $\D'$ is called an induced subgraph of $\D$. 
\end{defn}

In other words, we make $G'$ become an oriented subgraph graph of $\D$ by inheriting direction and weight from it. For a subset $S$ of $V(G)$, we write $\D\setminus S$ stands for the  subgraph of $\D$ induced on $G\setminus S$, i.e. the oriented graph obtained from $\D$ by deleting all vertices in $S$ and all edges incident to $S$. Note that $\D \setminus S$ is an induced subgraph of $\D$. If $S=\{v\}$ for some $v\in V(G)$, we write $\D\setminus v$ to mean $\D\setminus \{v\}$.

\medskip

\begin{rem} If $(x,y)\in E(\D)$ with $\omega(x)=\omega(y)=1$ then the ideal $I(\D)$ is not changed when adding the new directed edge $(y,x)$ into $\D$, so that in this paper we admit $(x,y)$ and $(y,x)$ into edges of $\D$ for any edge $xy\in E(G)$ with $\omega(x)=\omega(y)=1$.
\end{rem}

We next introduce the notion of a {\it reducible vertex} on an induced subgraph of $\D$ on an induced $5$-cycle of $G$, which plays a key role in the paper. 

\begin{defn} Let $C$ be an induced $5$-cycle of $G$ with $E(C) = \{xy,yz,zu,uv,vx\}$ and let $\C$ be the induced subgraph of $\D$ on $C$. The vertex $x$ is called a reducible vertex on $\C$ (see Figure \ref{BVertex}) if $\C\setminus x$ is unmixed and
\begin{enumerate}
\item either $(y,x), (v,x) \in E(\D)$, or
\item $\omega(x)=1$ and either $(y,x), (x,v),(u,v)\in E(\D)$ or $(v,x), (x,y),(z,y)\in E(\D)$.
\end{enumerate}
\end{defn}
\begin{figure}[h]
\includegraphics[scale=0.5]{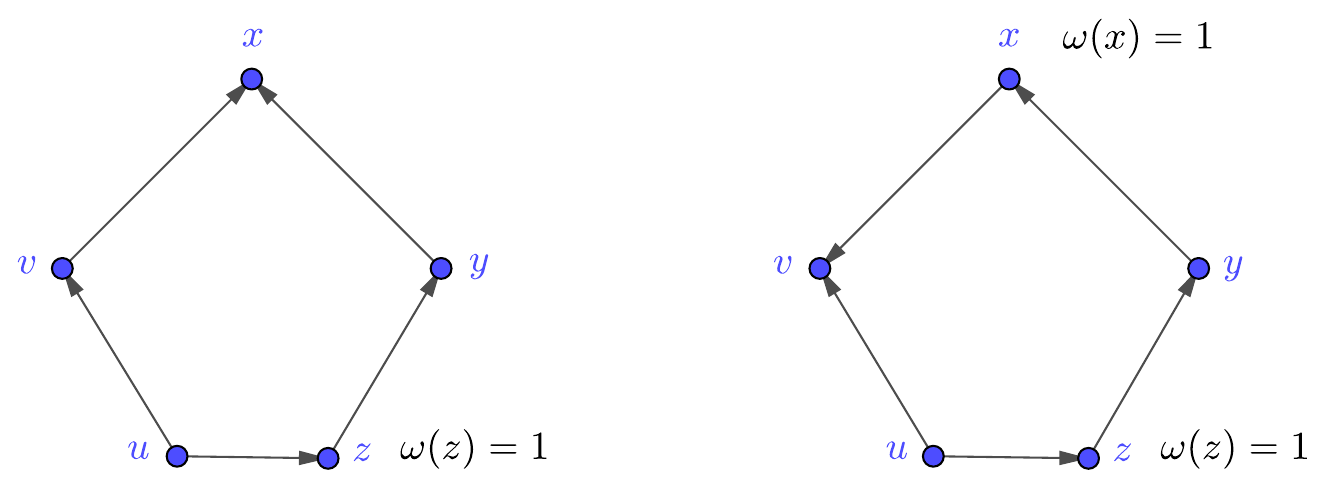}\\
\medskip
\caption{The vertex $x$ is reducible for each oriented graph.}
\label{BVertex}
\end{figure}

In the rest of this section, assume that the underlying graph $G$ of $\D$ is a $5$-cycle with the edge set $E(G)=\{xy,yz,zu,uv,vx\}$. Here we restate and prove a result given in \cite{PRT, SS} on classifying Cohen-Macaulay oriented graph $\D$ in a compact way
 (see \cite[Theorem 49]{PRT} and \cite[Theorem 4.5]{SS}). Note that $\D$ is unmixed if and only if  every strong vertex cover of $\D$ has just $3$ vertices due to Lemma \ref{unmixedness}.

\begin{lem}\label{CM-C5}  The following conditions are equivalent:
\begin{enumerate}
\item $\D$ is Cohen-Macaulay.
\item $\D$ is unmixed.
\item $\D$ has a reducible vertex.
\end{enumerate}
\end{lem}
\begin{proof}  In order to prove the lemma it suffices to show that $(2) \Longrightarrow (3)$ and $(3)\Longrightarrow (1)$.

$(2)\Longrightarrow (3)$. Assume that $\D$ has a sink vertex, say $x$. If $x$ is reducible, the it is a desired vertex. Assume that $x$ is not reducible. Without loss of general, we may assume that $(u,z)\in E(\D)$. Then, $\omega(z)\ne 1$ and $(z,y)\in E(\D)$. Observe that the vertex cover $\{v,x,y,z\}$ will be strong whenever $\omega(y)\ne 1$, and hence $\omega(y)=1$. In particular, $y$ is reducible.

Now we consider the case $\D$ has no sink vertex, i.e. $\D$ is an oriented cycle. We may assume that the orientation is $x\to y\to z\to u \to v\to x$. If $\D$ has an edge such that the weight of each endpoint is bigger than $1$. In this case, $\{z,y,x,v\}$ is a strong vertex cover of $\D$, a contradiction. Therefore, each edge of $\D$ has at least one endpoint with weight $1$. It implies that $\D$ has at least $3$ vertices of weight $1$. Consequently, $\D$ has an edge, say $(x,y)$, such that $\omega(x)=\omega(y)=1$. Consequently, $(y,x)\in E(\D)$ by our convention, and so $x$ is a sink vertex, a contradiction. Thus, we conclude that $\D$ has a reducible vertex.

$(3) \Longrightarrow (1)$. Let $x$ be a reducible vertex of $\D$. By symmetry, we consider two following cases:

\medskip

{\it Case $1$}: $(x,y),(x,v)\in E(\D)$. We may assume that $(u,z)\in E(\D)$. Let $m=\omega(x)$ and $n=\omega(z)$. Since $\D \setminus x$ is Cohen-Macaulay by Corollary \ref{corP4}, so is $I(\D)+(x^m) = I(\D\setminus x)+(x^m)$. On the other hand, $I(\D) \colon x^m  = (y,v,uz^n)$ is Cohen-Macaulay. By Lemma \ref{CM-Q} we conclude that $\D$ is Cohen-Macaulay. 

\medskip

{\it Case $2$}: $\omega(x)=1$ and $(y,x), (x,v), (u,v)\in E(\D)$. By the argument as in the previous case,  $I(\D) +(x)$ is Cohen-Macaulay. Observe that $I(\D)\colon x = (y,v^p,f)$ where $p=\omega(v)$ and $f = m_{\D}(uz)$.

It follows that $I(\D)\colon x$ is Cohen-Macaulay, and so $I(\D)$ is Cohen-Macaulay by Lemma \ref{CM-Q}, as required.
\end{proof}

\begin{lem}\label{CM-BV1} If $\D$ and $\D\setminus x$ are unmixed, then
$\D$ has a reducible vertex not adjacent to $x$ (maybe just $x$). 
\end{lem}
\begin{proof} Assume that $x$ is not reducible so that we must show that either $z$ or $u$ is reducible. Since $x$ is not reducible, we need to consider two following possible cases:

\medskip

{\it Case $1$}: $(v,x), (x,y) \in E(\D)$. Assume on the contrary that both $u$ and $z$ are not reducible. Then, we would have $\omega(x)\ne 1, \omega(y)\ne 1$ and $(y,z)\in E(\D)$. But in this case, $C=\{x,y,z,u\}$ is a strong vertex cover of $\D$, a contradiction. Thus, either $z$ or $u$ is reducible.

\medskip

{\it Case $2$}: $(x,y), (x,v)\in E(\D)$. Assume on the contrary that both $u$ and $z$ are not reducible. Then, we would have $\omega(y)\ne 1, \omega(v)\ne 1$ and $(y,z), (v,u)\in E(\D)$. But in this case, $\{y,z,u,v\}$ is a strong vertex cover of $\D$, a contradiction. Thus, either $z$ or $u$ is reducible, and the lemma follows.
\end{proof}

\begin{defn} The vertex $x$ is called a reducible vertex of the second kind on $\D$, if $(x,v), (u,v), (x,y),(z,y)\in E(\D)$.
\end{defn}

\begin{lem}\label{CM-BV2} If $\D$, $\D\setminus x$ and $\D\setminus u$ are unmixed, then either $x$ or $u$ is a reducible vertex or a reducible vertex of the second kind.
\end{lem}
\begin{proof} By symmetry, we may assume that $(y,z)\in E(\D)$. Now assume on the contrary that $x$ and $u$ both are neither reducible nor reducible in the second kind. Together with Lemma \ref{CM-BV1}, this fact implies that $y$ and $z$ both are reducible.

Assume that $(v,x)\in E(\D)$. Then, $(x,y) \in E(\D)$ as $x$ is not reducible. Since $z$ is reducible, we have $\omega(x)=1$. It follows that $x$ is reducible, a contradiction. Hence, $(v,x)\notin E(\D)$, and so $(x,v)\in E(\D)$.

Assume that $(v,u)\in E(\D)$. Then, $(u,z) \in E(\D)$ as $u$ is not reducible. It follows that $\omega(u)=\omega(z)=1$ since $I(\D\setminus y)$ and $I(\D\setminus x)$ are unmixed, respectively. It follows that $(z,u)\in E(\D)$, and so $u$ is reducible, a contradiction. Hence, $(u,v)\in E(\D)$.

So far we have $(u,v), (x,v), (z,y)\in E(\D)$. Next if $(y,x)\in E(\D)$, then $\omega(x)=1$ as $I(\D\setminus u)$ is unmixed. But in this case, $x$ is reducible, a contradiction.  It implies that $(x,y)\in E(\D)$. Therefore, $x$ is a reducible vertex of the second kind, a contradiction.

In summary, the our hypothesis is not true, so that either $x$ or $z$ is a reducible vertex or a reducible vertex of the second kind, as required.
\end{proof}

\section{Cohen-Macaulay weighted oriented graphs}

This section is devoted to classify Cohen-Macaulay weighted oriented graphs $\D$ whose underlying graph $G$ has girth at least $5$.  If $\D$ is Cohen-Macaulay, then $G$ is in the class $\pc$. Thus, it is natural to study the weight $\omega$ on pendant edges and basic $5$-cycles of $G$.

The following lemma allows us to investigate the weight on pendant edges. 

\begin{lem} \label{CM-L01} Let $I$ be a proper unmixed monomial ideal of $R$ such that every generator of $I$ has the form $u^a w^b$ where $u,w$ are variables (maybe $u=w$) and $a,b\in \N$. Let $f$ be a monomial such that $f\notin I$. Assume that $x^my^p$ and $x^nz^q$ are among  minimal generators of $I \colon f$ where $x,y,z$ are distinct variables and $m,n,p,q \in \Z_{>0}$.  Assume that $x^k \notin I \colon f$ for every $k$.  If $y$ does not appear in any minimal generator of $I\colon f$ except $x^my^p$, then $m\geqslant n$.
\end{lem}
\begin{proof} Now assume on the contrary that $m < n$.  Let
$$I \colon f = Q_1\cap Q_2\cap \cdots\cap Q_s$$
be an irredundant primary decomposition of $I$ with $\sqrt{Q_i} = P_i$ for $i=1,\ldots,s$.  Since $I$ is unmixed, so is $I\colon f$, and hence
$\sqrt{I\colon f} = P_1\cap P_2\cap \cdots\cap P_s$  is the irredundant primary decomposition of $\sqrt{I\colon f}$. Observe that $\sqrt{I\colon f}$ is generated by variables and square-free quadratic monomials, thus we may assume that
$$\sqrt{I\colon f} = (x_1,\ldots,x_t) + I(H)$$
where $H$ is a simple graph with $V(H) \subseteq V \setminus \{x_1,\ldots,x_t\}$.  In particular,  every minimal associated prime of $\sqrt{I\colon f}$ has the form $(x_1,\ldots,x_t) + (v\mid v\in C)$ where $C$ is a minimal vertex cover of $H$.  By the assumption,  $xy$ is an edge of $H$ with $\deg_H(y)=1$. Hence,  every maximal independent set of $H$ must contain either $x$ or $y$ but not both. Thus, every associated prime of $I\colon f$ contains either $x$ or $y$ but not both. 

Assume that $x\in P_j$ for $j=1,\ldots,r$ and $y\in P_j$ for $j = r+1,\ldots,s$ where $1\leqslant r < s$. Write $I = I_1\cap I_2$ where
$$I_1 = Q_1\cap \cdots\cap Q_r \text{ and } I_2 = Q_{r+1}\cap \cdots\cap Q_s.$$
Then, $y$ does not appear in any generator of $I_1$ and $x$ does not appear in any generator of $I_2$. Since $x^my^p$ is a generator of $I$, we deduce that $x^m$ is a generator of $I_1$.

Now because $x^nz^q$ is a generator of $I\colon f$, then $x^nz^q = \lcm(g,h)$ where $g$ is a generator of $I_1$ and $h$ is a generator of $I_2$. Since $x\nmid h$, it follows that $x^n \mid g$. Because $x^m$ is a generator of $I_1$, it follows that $m\geqslant n$, and lemma follows.
\end{proof}

\medskip

We now move on to investigate the weight on subgraph of $\D$ induced on a basic $5$-cycle of $G$. In the following lemmas, assume that $\D$ is unmixed and the underlying graph $G$ is in the class $\pc$. Let $C$ be a $5$-basic cycle of $G$ with $E(C) = \{xy,yz,zu,uv,vx\}$ (see e.g. Figure \ref{ABSet}) and let $\C$ be the induced oriented subgraph of $\D$ on $C$. 

\medskip

We start with a combinatorial fact on the graph $G$. Here, for a subset $S$ of $G$, we define the closed neighborhood of $S$ in $G$ by
$N[S] = S \cup \{u \mid uv\in E(G) \text{ for some } v \in E(G)\setminus S\}$. If $S=\{v_1,\ldots,v_k\}$, we write $N[v_1,\ldots,v_k]$ to mean $N[S]$.

\begin{lem}\label{CM-L02}  Assume that $\deg_G(x)>2$ and $N(x) = \{y,v,x_1,\ldots,x_m\}$.  Then, 
\begin{enumerate}
\item there is an independent set of $G$ with $m$ vertices,  say $\{a_1,\ldots,a_m\}$,  such that
\begin{enumerate}
\item $G[x_1,\ldots,x_m,a_1,\ldots,a_m]$ consists of $k$ disjoint edges $x_1a_1, \ldots, x_m a_m$.
\item $N[a_1,\ldots,a_m] \cap V(C) = \emptyset$.
\end{enumerate}
\item Assume further that $\deg_G(z) > 2$ and $N(z) = \{y,u,z_1,\ldots,z_n\}$. Then, there is an independent set of $G$ with $n$ vertices,  say $\{b_1,\ldots,b_n\}$,  such that
\begin{enumerate}
\item $G[z_1,\ldots,z_n,b_1,\ldots,b_n]$ consists of $n$ disjoint edges $z_1b_1, \ldots, z_n b_n$.
\item $N[b_1,\ldots,b_n] \cap V(C) = \emptyset$.
\item $\{a_1,\ldots,a_m,b_1,\ldots,b_n\}$ is an independent set of $G$.
\end{enumerate}
\end{enumerate}
\end{lem}
\begin{proof} $(1)$ We construct recursively the set $\{a_1,\ldots,a_m\}$ as follows.  Assume that we have collected $a_1,\ldots,a_{j-1}$ for $j < k$.  Since $G$ is in the class $\pc$,  we have $x_j$ is either lying in a pendant edge or a $5$-basic cycle.  In the former case,  we choose $a_j$ to be the leaf of this pendant edge.  In the second case,  there is a $5$-basic cycle,  say $C_j$,  containing $x_j$.  Then,  we choose $a_j$ to be a neighbor of $x_j$ and adjacent to another vertex of degree $2$ in this cycle (see Figure \ref{ABSet}). Since $G$ has no cycle of length $4$,  $a_j \ne a_i$ for all $i < j$. Because, if it is not the case then $G$ has a $4$-cycle $x\to x_i\to a_i\to x_j\to x$, a contradiction. By the same argument, $a_j$ is not adjacent to any $x_i$ for $i < j$.

Now since $C_j$ and $C$ are disjoint and $N[a_j] \subseteq V(C_j)$, we have $N[a_j] \cap V(C) = \emptyset$, and hence $N[a_1,\ldots,a_j] \cap V(C)=\emptyset$.

\begin{figure}[h]
\includegraphics[scale=0.5]{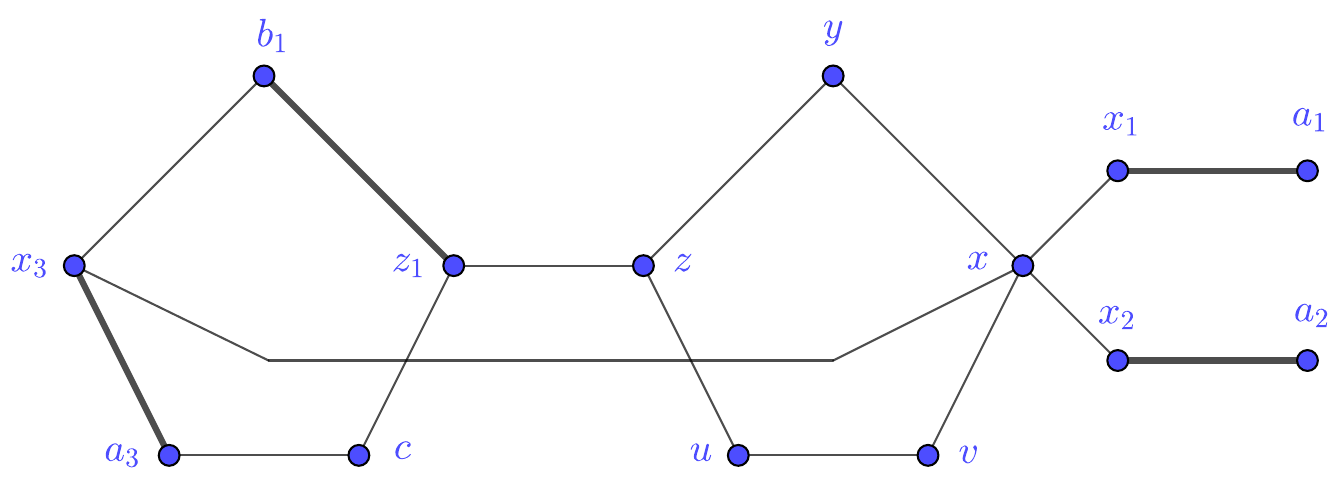}\\
\medskip
\caption{The vertices $a_i$ and $b_j$.}
\label{ABSet}
\end{figure}

Now we verify that $a_j$ is not adjacent to $a_i$ for any $i < j$.  Indeed,  assume on the contrary that $a_j$ is adjacent to $a_i$ for some $i < j$.  Then,  $a_j$ and $a_i$ are not leaves,  hence by the construction,  they are  in the basic $5$-cycles $C_j$ and $C_i$,  respectively.  Since $a_j$ and $a_i$ are adjacent to each other,  $C_i =C_j$. In this case, $x_i$ and $x_j$  are vertices of $C_j$ as well. Let $w$ be the vertex of $C_j$ which is adjacent to both $x_i$ and $x_j$.  Then,  $G$ has a $4$-cycle $x \to x_i\to w \to x_j\to x$, a contradiction.  Thus,  $a_j$ is not adjacent to any vertex $a_i$ with $i < j$. Recall that $a_j$ is not adjacent to any $x_i$ for $i < j$, so $G[x_1,\ldots,x_j, a_1,\ldots,a_j]$ consists of $j$ disjoint edges $x_1a_1,\ldots, x_j a_j$.

By using this process,  we finally obtain the desired set $\{a_1,\ldots,a_k\}$.

$(2)$ We proceed to choose $b_j$ as follows. If $z_j$ is lying in a pendant edge, then $b_j$ is the leaf of this pendant edge. In the other case, $z_j$ must be in a basic $5$-cycle, say $C'$. If $C'$ does not contain any vertex in $\{x_1,\ldots,x_m\}$, then choose $b_j$ as in the previous case. If $C'$ contains some vertex in this set, say $x_i$, then we choose $b_j$ to the the vertex in $C'$ adjacent to both $x_i$ and $z_j$ (see Figure \ref{ABSet}). 

By the same argument as in the previous case, we get the properties $(a)$ and $(b)$. In order to prove the property $(c)$, assume on the contrary that $a_i$ and $b_j$ are adjacent to each other for some $i$ and $j$. In particular, neither of them is a leaf, so $x_i$ and $y_j$ are lying on the same $5$-basic cycle $C_i$. By the choice of $a_i$ and $b_j$, we have $N_G(b_j) = \{x_i,z_j\}$ and $N_G(a_i) \ne \{x_i,z_j\}$. Therefore, $a_i$ and $b_j$ are not adjacent to each other, a contradiction, and the property $(c)$ follows. The proof of the lemma is complete.
\end{proof}

\begin{lem} \label{CM-L03} $\C$ is unmixed.
\end{lem}
\begin{proof} If $\D=\C$, there is nothing to show. If not then, assume on the contrary that $\C$ is not unmixed so that it has a strong vertex cover, say $F$, with $L_3(F)\ne\emptyset$. Let $S_2= L_2(F)$ and $S_3 = L_3(F)$ with respect to the oriented graph $\C$. We  have two possible cases:

{\it Case $1$}: $\C$ has just one vertex of degree at least $3$ and assume that $x$ is such a vertex. Let $N(x) = \{y,v,x_1,\ldots,x_m\}$. By Lemma \ref{CM-L02}, there is an independent set of $G$ with $m$ vertices, say $\{a_1,\ldots,a_m\}$, such that 
$N[a_1,\ldots,a_m]\cap V(\C)=\emptyset$ and $G[x_1,\ldots,x_m,a_1,\ldots,a_m]$ consists of disjoint edges $x_1a_1,\ldots,x_ma_m$. Since $(V(\C)\setminus F) \cup \{a_1,\ldots,a_m\}$ is an independent set of $G$, hence there is a maximal independent set of $G$, say $S$, contains it. In particular, $U = V(G)\setminus S$ is a minimal vertex cover of $G$ and $U\cap V(\C) \subseteq F$. Let $U' = U \cup F$ so that $U'$ is also a vertex cover of $G$.  We now prove two following claims: 

{\it Claim $1$}: $L_3(U') = S_3$. For the inclusion $L_3(U') \subseteq S_3$, let $a\in L_3(U')$. If $a\notin F$, then $a\in U\setminus V(\C)$ since $U\cap V(\C) \subseteq F$. Note that $L_3(U)=\emptyset$ by \cite[Proposition 6]{PRT}, so $a$ is adjacent to some vertex, say $b$, of $S$. On the other hand, $b\in N_G(a)\subseteq U'$, so  $b\in F$ because $b\notin U$. It follows that $b=x$ since $a\notin V(\C)$. In particular, $a = x_j$ for some $j$, and hence  $a$ is adjacent to $a_j\in S\setminus V(\C)$. This implies $N_G(a)\not \subseteq U'$, a contradiction. Thus, $a\in F$.

With this assumption, if  $a\ne x$, then $a\in S_3$ since $N_G(a) = N_{\C}(a)\subseteq F$. If $a = x$, since $x_1,\ldots,x_m\in U$, we have
$N_G(a)=N_G(x) = N_{\C}(x) \cup \{x_1,\ldots,x_m\} \subseteq F \cup U = U'$. Hence, $a\in S_3$, and hence  $L_3(U') \subseteq S_3$ holds. 

For the reverse inclusion $S_3\subseteq L_3(U')$, let $a\in S_3$. If $a\ne x$, then $N_G(a) = N_{\C}(a) \subseteq F$, and then $a\in L_3(U')$. If $a = x$, then $N_{G}(a) =N_{\C}(x) \cup \{x_1,\ldots,x_m\} \subseteq U'$ since $N_{\C}(x)\subseteq F$ and $x_1,\ldots,x_m\in U$. Hence, $a\in L_3(U')$, and the claim follows.

{\it Claim $2$}: $S_2\subseteq L_2(U')$. Indeed, if  $S_2=\emptyset$, there is nothing to show. If not then, let $a\in S_2$  and $b\in V(\C)\setminus F$ such that $(b,a)\in E(\C)$. Since $F$ has at least $4$ vertices, it forces $F = V(\C)\setminus\{b\}$. Together with the fact that $x_1,\ldots,x_m\in U'$, it implies that $(b,a)$ is the unique directed edge of $\D$ which is incident to the vertex $a$ and one vertex in $V(\D) \setminus U'$,  so $a\in L_3(U')$, and the inclusion follows.
\medskip

Since $F$ is a strong vertex cover of $\C$, two claims above yield  $U'$ is a strong vertex cover of $\D$. Since $L_3(U') = S_3\ne \emptyset$, $\D$ is not unmixed by Lemma \ref{unmixedness}, a contradiction. 

\medskip

{\it Case $2$}: $\C$ has two vertices of degree at least $3$ and assume that $\deg_G(x) > 2$ and $\deg_G(z) > 2$. Let $N_G(x) = \{y,v,x_1,\ldots,x_m\}$ and $N_G(z) = \{u,v,z_1,\ldots,z_n\}$. By Lemma \ref{CM-L02}, we have an independent sets of $G$, say $S'=\{a_1,\ldots,a_m, b_1,\ldots,b_n\}$, such that $G[S'] \cap V(\C)=\emptyset$ and $G[S', x_1,\ldots,x_m,z_1,\ldots,z_n]$ consists of disjoint edges $x_1a_1,\ldots,x_ma_m$, $z_1b_1,\ldots,z_nb_n$.

Now since $(V(\C)\setminus F)\cup S'$ is an independent set of $G$, we  extend it to a maximal independent set of $G$, say $S$. Then, $U = V(G)\setminus S$ is a minimal vertex cover of $G$. Now let $U' = U \cup F$ so that $U'$ is a vertex cover of $\D$. By the same argument as in the previous case, we can verify that $L_3(U') = S_3$ and $S_2\subseteq L_2(U')$. It follows that $U'$ is a strong vertex cover of $\D$, a contradiction.

In summary, $L_3(C') = \emptyset$, and the lemma follows from Lemma \ref{unmixedness}.
\end{proof}

\begin{lem} \label{CM-L04} If $\omega(x) \ne 1$ and $(w,x)\in E(\D)$ for some $w\in N_G(x)\setminus\{y,v\}$, then $(y,x), (v,x)\in E(\D)$.
\end{lem}
\begin{proof} Let $N(x) = \{y,v,x_1,\ldots,x_m\}, \text{ where } m \geqslant 1$, with $(x_1,x)\in E(D)$. Now assume on the contrary that $(y,x) \notin E(\D)$ or $(v,x)\notin E(\D)$. We may assume that $(v,x)\notin E(\D)$ so that $(x,v)\in E(\D)$ (see Figure \ref{InCycle}). 

\begin{figure}[h]
\includegraphics[scale=0.45]{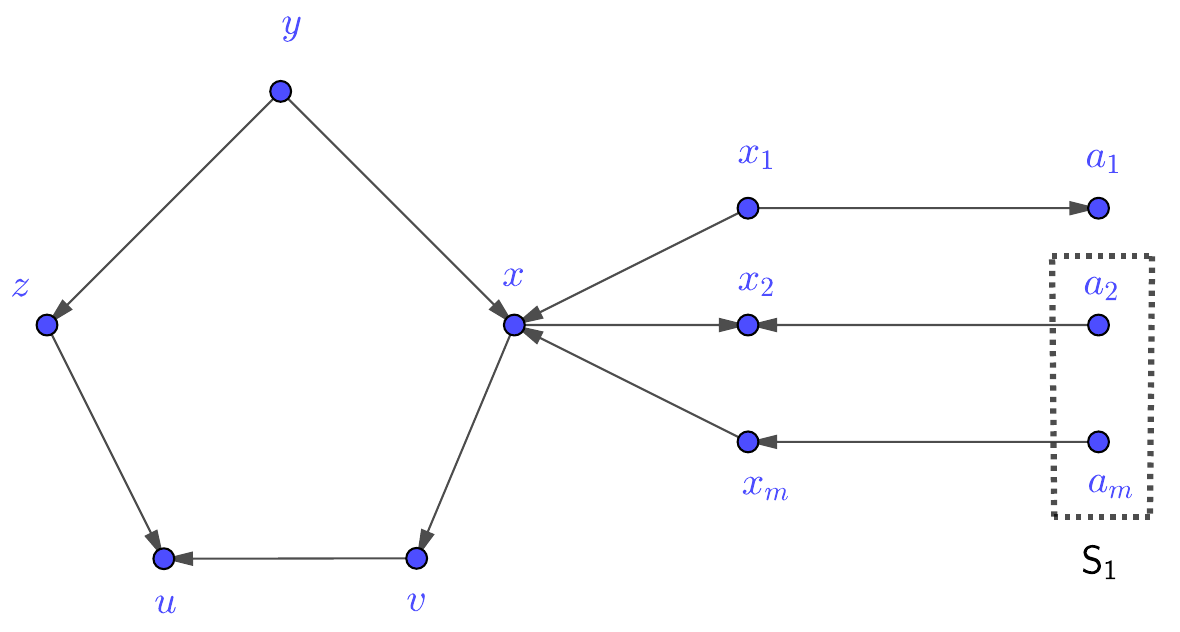}\\
\medskip
\caption{The structure of $G$.}
\label{InCycle}
\end{figure}

By Lemma \ref{CM-L02},  there is an independent set $\{a_1, \ldots,a_m\}$ of $G$ such that the graph $G[x_1,\ldots,x_m,a_1,\ldots,a_m]$ consists of disjoint edges $x_1a_1,\ldots, x_ma_m$ and $N[a_1,\ldots,a_m] \cap V(C) = \emptyset$. Let $S_1 = \{a_2,\ldots,a_m\}$.  As $\girth(G) \geqslant 5$, we deduce that $\{x_1, z, v\} \cup S_1$ is an independent set of $G$.  Now extend this set to a maximal independent set of $G$,  say $S$.  Then, $C = V(G) \setminus S$ is a minimal cover of $G$ so that $|C| = \hgt(I(\D))$. Note that $x,y,u,x_2, \ldots,x_m \in C$ and $x_1, z,v \notin C$.   

Let $C' = C\cup \{v\}$. Then, we have $L_3(C') = \{v\}$. Observe that $N(x) \setminus C' = \{x_1\}$ and  $(x_1,x)\in E(\D)$, so that $x\in L_2(C')$. Together with $\omega(x)\ne 1$ and $(x,v)\in E(\D)$, this fact yields $C'$ is a strong vertex cover of $\D$. In particular, $\D$ is not unmixed by Lemma \ref{unmixedness}, a contradiction. Hence, $(v,x)\in E(\D)$. By the same way, $(y,x)\in E(\D)$,  and the lemma follows.
\end{proof}

\begin{lem}\label{CM-L05} If $\deg_G(x) > 2$, then $\C\setminus x$ is unmixed.
\end{lem}
\begin{proof} In order to prove the lemma, we consider three possible cases:

\medskip

{\it Case $1$}: $\omega(x)=1$. Let $w\in N_G(x)\setminus \{y,v\}$. If $(x,w)\in E(\D)$, then $xw^k\in I(\D)$ where $k = \omega(w)$. It follows that
$$I(\D) \colon w^k = (x, f_1,f_2,f_3, \ldots)$$
where $f_1,f_2,f_3$ are monomials corresponding to edges of $\D$ induced on edges $yz,zu,uv$ of $G$, respectively. Moreover, we can verify that
\begin{enumerate}
\item $f_1,f_2,f_3$ are among minimal generators of $I(\D) \colon w^k$;
\item all other generators of $I(\D)$ contain neither $y$ nor $v$;
\item any power of either $z$ or $u$ is not in $I(\D)\colon w^k$.
\end{enumerate}
Together with Lemma \ref{CM-L01}, we obtain
$$\deg_z(f_2)\leqslant \deg_z(f_1), \text{ and } \deg_u(f_2) \leqslant \deg_u(f_3).$$
Thus, $I(\C\setminus x)$ is unmixed by Corollary \ref{corP4}.

Similarly, in the case $(w,x)\in E(\D)$, we prove $I(\C\setminus x)$ is unmixed by considering the colon ideal $I(\D)\colon w$.
\medskip

{\it Case $2$}: $\omega(x) \ne 1$ and $(w,x)\in E(\D)$ for some $w\in N_G(x)\setminus\{y,v\}$. Then, $(y,x), (v,x)\in E(\D)$ by Lemma \ref{CM-L03}. The proof now is the same as in the previous case by considering the colon ideal $I(\D) \colon w$.

\medskip

{\it Case $3$}: $\omega(x)\ne 1$ and $(x,w)\in E(\D)$ for every $w\in N_G(x)\setminus\{y,v\}$. Let $w$ be such a vertex so that $xw^k\in I(\D)$ where $k=\omega(w)$. The proof is the same as in the previous case by considering the colon ideal $I(\D) \colon w^k$, and the lemma follows.
\end{proof}

We now in position to prove the main result of the paper.

\begin{thm}\label{main-theorem} Let $\D$ be a weighted oriented graph with underlying graph $G$ of girth at least $5$. Then, the following conditions are equivalent:
\begin{enumerate}
\item $\D$ is Cohen-Macaulay.
\item $G$ is Cohen-Macaulay and $\D$ is unmixed.
\item $G$ is in the class $\pc$ and $\D$ satisfies:
\begin{enumerate}
\item For every pendant edge $xy$ of $G$ with $\deg_G(y)=1$ and $\omega(x) \ne 1$, $(y,x)\in E(\D)$ if $(z,x)\in E(\D)$ for some $z\in N(x)\setminus \{y\}$.
\item If $\C$ is a subgraph of $\D$ induced on a basic $5$-cycle of $G$, then
\begin{enumerate}
\item $\C$ is unmixed.
\item If $x$ is a vertex on $\C$ with $\deg_G(x) > 2$, then $\C\setminus x$ is unmixed.
\item If $x$ is a vertex on $\C$ such that $\omega(x)\ne 1$ and $(w,x)\in E(\D)$ for some $w\in V(\D)\setminus V(\C)$, then $(y,x), (v,x)\in E(\D)$, where $N_{\C}(x)=\{y,v\}$.
\end{enumerate}
\end{enumerate}
\end{enumerate}
\end{thm}

\begin{proof} $(1)\Longrightarrow (2)$ Since $\D$ is Cohen-Macaulay, we have $\D$ is unmixed. Since $I(G)=\sqrt{I(\D)}$, so that $G$ is Cohen-Macaulay by \cite[Theorem 2.6]{HTT}.

$(2)\Longrightarrow (3)$ Since $G$ is a Cohen-Macaulay graph of girth at least $5$, $G$ is in the class $\pc$ by Lemma \ref{HMT}.  If  $G$ is just a $5$-cycle, there is nothing to show. If not then, we restate the property $(a)$ as:  for every pendant edge $xy$ of $G$ with $\deg_G(y)=1$, 
$m_{\D}(xy) \geqslant m_{\D}(xz)$ for any $xz\in E(G)$. In this form, the property $(a)$ follows from Lemma \ref{CM-L01}. The properties $(i)$-$(iii)$ follow from Lemmas \ref{CM-L03},  \ref{CM-L04} and \ref{CM-L05}.

$(3)\Longrightarrow (1)$ We prove by induction on the number of basic $5$-cycles of $G$.  If $G$ has no basic $5$-cycle, then its pendant edges form a perfect matching in $G$. In this case, together with the condition $(a)$, Lemma \ref{H2}  yields $\D$ is Cohen-Macaulay.

Assume that $G$ has some basic $5$-cycles. If $G$ is just a $5$-cycle, the desired implication follows from Lemma \ref{CM-C5}. Assume that $G$ is not a $5$-cycle. Let $C_1,\ldots,C_r$ be the basic $5$-cycles of $G$ with $r\geqslant 1$ and let $P$ be the set of pendant edges of $G$.  We may assume that $E(C_1) = \{xy,yz,zu,uv,vx\}$ and $N_G(x)=\{y,v,x_1,\ldots,x_k\}$ with $k\geqslant 1$. 

For example, consider $\D$ as illustrated in Figure \ref{orientedGraph}, the underlying graph $G$ has two basic $5$-cycles $C_1$ and $C_2$, where $C_1$ as above, and $C_2$ has the vertex set $V(C_2) = \{b_1,z_1,c,a_3,x_3\}$. Note also that $G$ has pendant edges $P = \{x_1a_1, x_2a_2\}$.

\begin{figure}[h]
\includegraphics[scale=0.5]{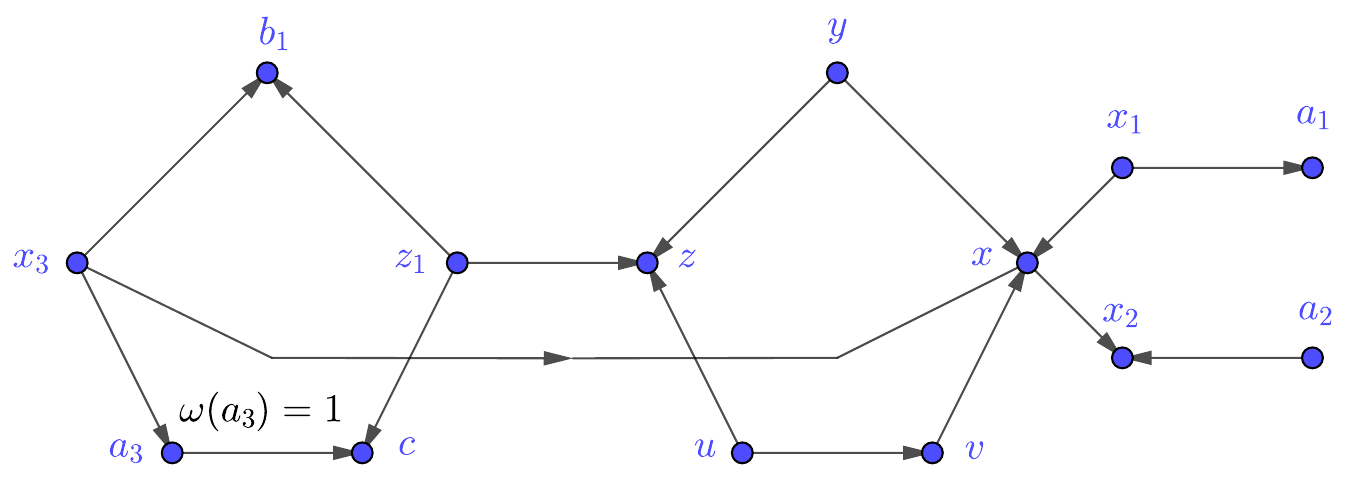}\\
\medskip
\caption{The oriented graph $G$.}
\label{orientedGraph}
\end{figure}

Now we come back to prove the theorem. Let $\C_1$ be the induced subgraph of $\D$ on $C_1$. For convenient, if $G_j$, with some index $j\in\Z_{>0}$, is a subgraph of $G$, then we use the symbol $\G_j$ to denote the induced subgraph of $\D$ on $G_j$.

\medskip

{\it Case $1$}: $C_1$ has two vertices of degree at least $3$. We may assume that $\deg_G(x) \geqslant 3$ and $\deg_G(z) \geqslant 3$. Since $\C_1$ is unmixed, by Lemma \ref{CM-BV2} we reduce to three possible subcases:
\medskip

{\it Subcase $1.1$}: $x$ is a reducible vertex on $\C_1$ and $(y,x),(v,x)\in E(\D)$ (see Figure \ref{orientedGraph}). Let $G_1$ be the graph obtained from $G$ by removing  two edges $yx$ and $vx$. Let $m=\omega(x)$ so that $yx^m,vx^m\in I(\D)$. Then,
$$I(\D) \colon x^m = (y, v) +I(\D\setminus\{y,v\})\colon x^m$$
and
$$I(\D)+( x^m) = (x^m, I(\G_1)).$$

We next prove these ideals are Cohen-Macaulay. Let $w$ be a new vertex and let $H$ be the graph with $V(H) =V(G_1)\cup\{w\}$ and $E(H)=E(G_1) \cup \{xw\}$. In other words,  $H$ is obtained from $G_1$ by adding the new edge $xw$. Then, $H$ is in the class $\pc$ with $5$-basic cycles $C_2,\ldots,C_r$ and pendant edges $P'=P\cup \{zy,uv,xw\}$ where $y,v,w$ are leaves. 

\begin{figure}[h]
\includegraphics[scale=0.5]{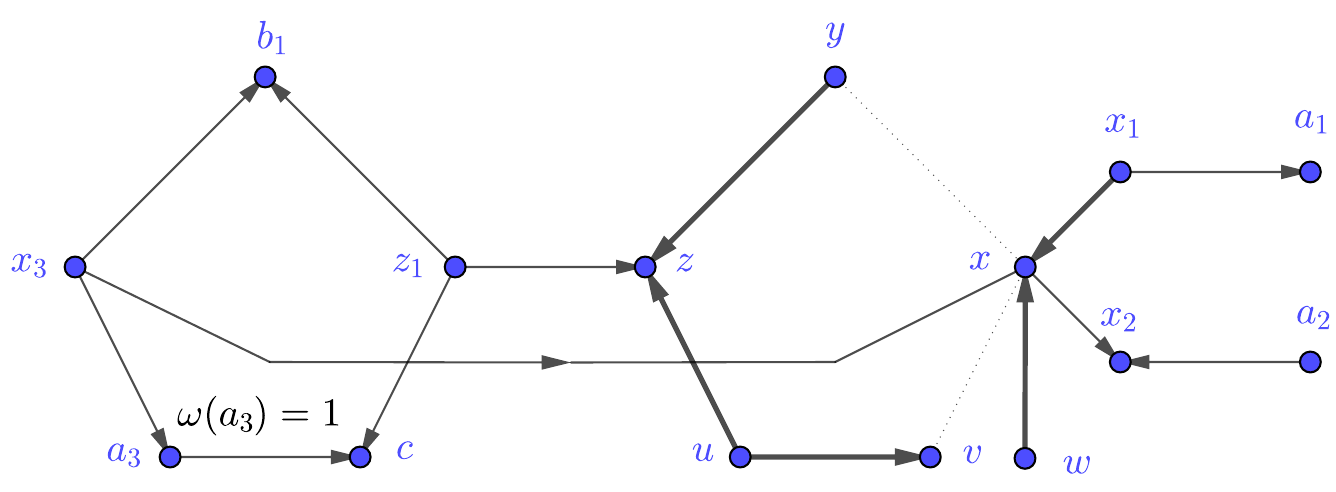}\\
\medskip
\caption{The oriented graph $\H1$.}
\label{Sub11}
\end{figure}

Let $\H1$ be the oriented graph whose underlying graph $H$ with the edge set $E(\H1) = \{(w,x)\} \cup E(\G_1)$ (see Figure \ref{Sub11}) and the weight inheriting from $\omega$, still denoted $\omega$, together with $\omega(w)=1$. We now verified the oriented graph $\H1$ satisfies the condition $(3)$. It suffices to prove the property $(a)$. In order to do this, by Lemma \ref{CM-L05}, it remains to verify this property for the pendant edge $zy$  when $\omega(z)\ne 1$.  Let $(z_1,z)\in E(\H1)$ for some $z_1\in V(H_1)\setminus \{y\}$. If $z_1=u$, then $(y,z)\in E(\H1)$ by Corollary \ref{corP4} since $\C_1\setminus x$ is unmixed. If $z_1 \ne u$, by applying Lemma \ref{CM-L04} we get $(y,z)\in E(\H1)$, and the property $(a)$ holds for $\H1$.

Since $H$ has $r-1$ basic $5$-cycles, $\H1$ is Cohen-Macaulay by the induction hypothesis. Since $(w,x)\in E(\H1)$, so that $wx^m \in I(\H1)$, by Lemma $\ref{CM-Q}$ we have $I(\H1) \colon w$ is Cohen-Macaulay. Note that $I(\H1) \colon w = (x^m, I(\G_1))$, so $I(\D)+(x^m)$ is Cohen-Macaulay.

We proceed to prove that $I(\D) \colon x^m$ is Cohen-Macaulay by the same manner. Let $G_2 = G\setminus \{y,v\}$ so that $\G_2 = \D\setminus \{u,v\}$. 

Let $w_2$ be a new vertex and let $H_2$ be the graph with $V(H_2) =V(G_2)\cup\{w_2\}$ and $E(H_2)=E(G_2) \cup \{xw_2\}$. In other words,  $H_2$ is obtained from $G_2$ by adding the new edge $xw_2$. Then, $H_2$ is in the class $\pc$ with $5$-basic cycles $C_2,\ldots,C_r$ and pendant edges $P''=P\cup \{zu, xw_2\}$ where $u,w_2$  are leaves. 

Let $\H1_2$ be the oriented graph with underlying graph $H_2$ by setting $E(\H1_2) = \{(w_2,x)\} \cup E(\G_2)$ and the weight inheriting from $\omega$, still denoted $\omega$, together with $\omega(w_2)=1$. We now verified the oriented graph $\H1_2$ satisfies the condition $(3)$. It suffices to prove the property $(a)$ for the pendant edge $zu$ when $\omega(z) \ne 1$. If there is $(z_1,z)\in E(\H1_2)$ for some $z_1\in E(\H1_2)\setminus\{u\}$. Since $z_1\in N_G(z)\setminus \{y,u\}$, by coming back the graph $\D$, we deduce that $(u,z)\in E(\D)$ by Lemma \ref{CM-L04}. Consequently, $(u,z)\in E(\H1_2)$, and the property $(a)$ holds for $\H1_2$.

Thus, $\H1_2$ is Cohen-Macaulay by the induction hypothesis. In particular, $I(\H1_2)\colon x^m$ is Cohen-Macaulay. Since $I(\H1_2)\colon x^m = (w_2) + I(\G_2) \colon x^m$,  we deduce that $I(\G_2) \colon x^m$ is Cohen-Macaulay too. Thus, $I(\D)\colon x^m = (y,v) + I(\G_2)\colon x^m$ is Cohen-Macaulay.

Since $\sqrt{I(\D)  +(x^m)} = (x, I(G\setminus x))$ is Cohen-Macaulay, so is $I(G\setminus x)$. Thus, $G\setminus x$ is well-covered.  As $x$ is not an isolated vertex,  it is a shedding vertex.   On the other hand,  $\sqrt{I(\D) \colon x^m} = (y,  v, I(G_x)) = I(G)\colon x$.  By Lemma \ref{dim}, we get $$\dim R/I(\D) = \dim R/I(\D)\colon x^m = \dim R/(I(\D),x^m).$$ 
It follows that $I(\D)$ is Cohen-Macaulay by Lemma \ref{CM-Q}.

The proofs of remaining cases are almost the same, so that we only sketch the proof.

\medskip

{\it Subcase $1.2$}: $x$ is a reducible vertex on $\C_1$ and $\omega(x)=1, (y,x), (x,v),(u,v) \in E(\D)$. We carry out the same argument as in the previous case by considering two ideals $I(\D)\colon x$ and $I(\D)+(x)$. The key point here is that $I(\D)\colon x = (y,v^n) + I(\D\setminus\{y,v\})\colon x$.

\medskip

{\it Subcase $1.3$}: $x$ is a reducible vertex of the second kind (e.g. the vertex $z_1$ on the cycle $C_2$ as in Figure \ref{orientedGraph} if we consider $C_2$ instead of $C_1$). Then,  $$(x,v),(x,y), (z,y),(u,v)\in E(\D).$$
The proof is carried out as in the subcase 1.1 by consider two ideals $I(\D) \colon x$ and $I(\D) + (x)$. The main point is that $I(\D)\colon x = (y^{\omega(y)},v^{\omega(v)})+I(\D\setminus\{y,v\})\colon x$ and $x$ is a source if  $\omega(x)\ne 1$ by the property $(iii)$. In this case, we assign the weight $1$ for $x$ before carrying out the proof.

{\it Case $2$}: $C_1$ has exactly one vertex of degree at least $3$ and assume that $x$ is such the vertex. The proof is carried out by the same manner as in the subcase 1.1 by consider two subcases. The first one is $x$ is a reducible vertex on $\C_1$. The second one is $\C_1$ has a reducible vertex not adjacent to $x$. We may assume $z$ is such a vertex. Then, the vertex $z$ now plays the role of $x$ in the proof of the subcase 1.1. The proof now is complete.
\end{proof}

\begin{exm}Let $\D$ be the oriented graph depicted in Figure \ref{orientedGraph}. In order to get a weight on this graph, assign any positive integer for each vertex except for the vertex $a_3$. Then, $\D$ is Cohen-Macaulay by Theorem \ref{main-theorem}.  Indeed, we need to show that the condition $(3)$ holds true. Actually, the conditions $(a)$ and $(iii)$ can verify directly. The conditions $(i)$ and $(ii)$ are not hard to check by using  Lemma \ref{CM-C5} and Corollary \ref{corP4}, respectively.
\end{exm}

We conclude the paper with a note on classifying Cohen-Macaulay oriented graphs.  In fact, Pitones, Reyes and Toled \cite{PRT} conjectured that (see \cite[Conjecture 53]{PRT}): $\D$ is Cohen-Macaulay if and only if the underlying graph $G$ is Cohen-Macaulay and $\D$ is unmixed. Unfortunately, this conjecture was disproved (see \cite[Section 4]{SSTY}). It is worth mentioning that the paper \cite{SSTY} is devoted to study the Cohen-Macaulay edge-weighted (undirected) graphs defined by Paulsen and Sather-Wagstaff  \cite{PS}. On the other direction, the characterization of sequentially Cohen-Macaulay edge-weighted graphs is studied in \cite{MDV}.

\medskip

Based on Theorem \ref{main-theorem} and our computations we propose the following conjecture. Recall that a graph of girth at least $4$, i.e. it contains no cycles of length $3$, is called a {\it triangle-free} graph.

\begin{conj} Let $\D$ be a weighted oriented graph whose underlying graph $G$ is triangle-free. Then, $\D$ is Cohen-Macaulay if and only if $G$ is Cohen-Macaulay and $\D$ is unmixed.
\end{conj}

% -----------------------------------------------------------
\subsection*{Acknowledgment}  This work is  partially supported by NAFOSTED (Vietnam) under the grant number 101.04-2023.36.

% -----------------------------------------------------------

\end{document}